\documentclass[10pt]{amsart}

\usepackage{hyperref}

\usepackage{amssymb}
\usepackage{amsmath}
\usepackage{graphicx}
\usepackage{tikz}
\usepackage{appendix}
\usepackage{bibentry}

\let\newpf\proof \let\proof\relax

\newcommand{\bt}{\begin{thm}}
\newcommand{\et}{\end{thm}}

\newcommand{\bl}{\begin{lemma}}
\newcommand{\el}{\end{lemma}}

\newcommand{\beq}{\begin{eqnarray}}
\newcommand{\eeq}{\end{eqnarray}}

\def\be{\begin{equation}}
\def\ee{\end{equation}}

\def\ba{{\begin{align}}}
\def\ea{{\end{align}}}

\newtheorem{thm}{Theorem}[section]

\newtheorem{cor}[thm]{Corollary}

\newtheorem{lemma}[thm]{Lemma}

\theoremstyle{remark}
\newtheorem{rem}{Remark}[section]

\numberwithin{equation}{section}

\def \bn {\hfill \\ \smallskip\noindent}

\theoremstyle{definition}

\def\proof{\bn {\bf Proof.} }

\def\note#1
{\marginpar

{\tiny $\leftarrow$
\par
\hfuzz=20pt \hbadness=9000 \hyphenpenalty=-100 \exhyphenpenalty=-100
\pretolerance=-1 \tolerance=9999 \doublehyphendemerits=-100000
\finalhyphendemerits=-100000 \baselineskip=6pt
#1}\hfuzz=1pt}

\newcommand{\dist}{\operatorname{dist}}

\renewcommand{\mod}{\operatorname{mod}}

\newcommand{\C}{{\mathbb C}}

\newcommand{\N}{{\mathbb N}}
\newcommand{\Q}{{\mathbb Q}}
\newcommand{\R}{{\mathbb R}}
\newcommand{\T}{{\mathbb T}}

\newcommand{\Z}{{\mathbb Z}}

\def\B0{{\bold{0}}}

\catcode`\@=12

\def\Empty{}
\newcommand\oplabel[1]{
  \def\OpArg{#1} \ifx \OpArg\Empty {} \else
  	\label{#1}
  \fi}

\newcommand{\comm}[1]{}
\newcommand{\comment}[1]{}

\begin{document}

\title{Full measure reducibility and localization for quasiperiodic Jacobi operators: a topological criterion}
\title[]{Full measure reducibility and localization for quasiperiodic Jacobi operators: a topological criterion}
\begin{abstract}
We establish a topological criterion for connection between reducibility to constant rotations and dual localization, for the general family of analytic quasiperiodic Jacobi operators. As a corollary, we obtain the sharp arithmetic phase transition for the extended Harper's model in the positive Lyapunov exponent region.
\end{abstract}

\setcounter{tocdepth}{1}
\author{Rui Han and Svetlana Jitomirskaya}
\maketitle

\section{Introduction}

In this paper we study the general class of Jacobi operators
\begin{align}
(H_{c}(\theta)u)_n=c(\theta+n\alpha)u_{n+1}+\tilde{c}(\theta+(n-1)\alpha)u_{n-1}+v(\theta+n\alpha)u_{n},
\end{align}
where
$c(\theta)=\sum_{k}\hat{c}_k e^{2\pi i k(\theta+\frac{\alpha}{2})}\in
C^{\omega}(\T)$,  $\tilde{c}(\cdot) \in C^{\omega}(\T)$,
$\tilde{c}(\theta)=\overline{c(\theta)}$ on $\T,$ and $v(\theta)=\sum_{k}\hat{v}_k e^{2\pi i k \theta}\in C^{\omega}(\T)$. 
We will assume $\hat{v}_k=\overline{\hat{v}_{-k}}$, $\hat{c}_k\in\R$. Such operators arise as effective Hamiltonians in a tight-binding description of a crystal subject to a weak external magnetic field, with $c,v$ reflecting the lattice geometry and the allowed electron hopping between lattice sites. The prime example, both in math and in physics literature, is the extended Harper's model, see (\ref{defehm}).
Notice that when $c(\theta)\equiv 1$ (this corresponds to the nearest neighbor hopping on a square lattice) we get the Schr\"odinger operator
\begin{align}
(H(\theta)u)_n=u_{n+1}+u_{n-1}+v(\theta+n\alpha)u_n.
\end{align}

The Aubry dual of $H_c$ is an operator $\tilde H_c$ defined by
\begin{align}
(\tilde{H}_c(x)u)_m=\sum_{m^{\prime}}d_{m^{\prime}}(c,v)(x)u_{m-m^{\prime}},
\end{align}
where $d_{m^{\prime}}(c,v)(x)=\hat{c}_{m^{\prime}}e^{2\pi i (x-\frac{m^{\prime}}{2}\alpha)}+\hat{v}_{-m^{\prime}}+\hat{c}_{-m^{\prime}}e^{-2\pi i (x-\frac{m^{\prime}}{2}\alpha)}$.

The Aubry duality can be explained by the magnetic nature and corresponding gauge invariance of operators $H_c$ \cite{mz} and has been formulated and explored on different levels, e.g. \cite{mz}, \cite{gjls}, \cite{aj2}.
The dynamical formulation of Aubry duality is an observation that if
 $\tilde H_{c}(\theta)$ has an eigenvalue at $E$ with respective eigenvector $\{u_n\}$, then, considering its Fourier transform, $u(x) : = \sum_{n\in\mathbb{Z}} u_n \mathrm{e}^{2 \pi i n x} \in L^2(\mathbb{T})\setminus \{0\}$ and letting
\begin{equation}
M_\theta(x) = \left(\begin{array}{c c} u(x) & u(-x) \\ \mathrm{e}^{-2 \pi i \theta} u(x - \alpha) & \mathrm{e}^{2 \pi i \theta} u(-(x-\alpha)) \end{array}\right)  ~\mbox{,}
\end{equation}
$M_\theta$ provides an $L^2$ semiconjugacy between the transfermatrix cocycle of $H_c$ and the rotation $R_\theta = \begin{pmatrix} \mathrm{e}^{2 \pi i \theta} & 0 \\ 0 & \mathrm{e}^{-2 \pi i \theta} \end{pmatrix} ~\mbox{.}$  For $\theta$ that are not $\alpha$-rational, $\det M_\theta(x)$ doesn't vanish for a.e. $x$ \cite{ajm}, leading to reducibility of the  transfermatrix cocycle of $H_c$ to a constant rotation $R_\theta$. In particular, pure point spectrum for a.e. $\theta$ of $\tilde H_{c}(\theta)$ leads to reducibility for cocycles of $H_c$ for a.e. $E$ with respect to the density of states \cite{puig}, \cite{ajm},
with the quality of reducibility governed by the rate of decay of
$u_n.$ As there are well developed methods to prove localization (thus exponential decay of the eigenfunctions) in various applications, this can be used to establish further interesting consequences \cite{aj2,ajm,rui}.

With the development of recent powerful methods \cite{ak,afk,arc} to
establish non-perturbative reducibility directly and independently of
localization for the dual model, the reverse direction: obtaining
localization for $\tilde H_c$ from reducibility of $H_c$, first used
in a more restricted form back  in \cite{blt},  started
gaining prominence. In the Schr\"odinger case, reducibility provides a
direct construction of eigenfunctions for the dual model (with the
decay governed by the quality of reducibility), so their completeness
becomes the main issue. This has been considered a nontrivial question
even for the almost Mathieu family. It had been conjectured for a long
time \cite{conj} that $\lambda=e^{\beta}$, where $\beta$ is the upper
rate of exponential growth of denominators of continued fractions
approximants to $\alpha$ (see (\ref{beta})), is the phase transition
line from purely singular continuous spectrum to pure point
spectrum. A combination of the almost reducibility conjecture
\cite{arc} and techniques of \cite{afk,hy,yz} led to establishing
reducibility throughout the dual of the entire conjectured
localization region, yet completeness of the resulting eigenfunctions
remained a problem. This was recently resolved in \cite{ayz} where the
authors used delicate quantitive information on the reducibility and
therefore dual eigenfunctions with certain rate of decay to prove the
pure point spectrum part of the conjecture. More recently, in
\cite{jk}, the authors obtained an elementary proof of complete
localization for the dual model under the assumption of only  $L^2$
degree $0$-reducibility (see definition in (\ref{degree0}) of the Schr\"odinger cocycle for $H(\theta)$ for almost all energies with respect to the density of states measure.

For the Jacobi case the situation is more problematic.  It was noticed (albeit in a different form) in \cite{mz} that for $c\not\equiv 1$ the existence of reducibility at $E$ for the cocycle of $H_c$ may not lead to $E$ being an eigenvalue of $\tilde H_{c}$. 
The difficulty is also reflected in the extended Harper's model (see Section \ref{ehmsec}).
On the positive side, in the dual regions I and II, we do in general have purely absolutely continuous spectrum, which is always associated to reducibility, in region II, and pure point spectrum in region I \cite{JKS, ajm}. 
However on the negative side, purely absolutely continuous spectrum for a.e. $\theta$ has been proved throughout the whole self-dual region III in the anisotropic case \cite{ajm}. 
Thus whether reducibility implies localization for the dual model could depend on $c,v,\alpha$, and even the existence of dual eigenvectors, automatic in the Schr\"odinger case, becomes an issue.

In this paper, we answer  this question for analytic
$c$. We establish an if-and-only-if topological criterion in terms of
the function $c$ only, for the reducibility for $H_c$ to imply pure
point spectrum of $\tilde H_c$.  Thus we extend the result of
\cite{jk} to the Jacobi setting in a sharp way and also describe
exactly what happens in the region to which it does not extend. 
It turns out the {\it winding number} $w(c)$ of $c(\theta)$  (see (\ref{winding})) is the key quantity.

While in this paper we are dealing with $C^\omega$ cocycles, and in
our main application we will have analytic reducibility as an input, the
general theorem only requires  a much weaker reducibilty, namely  $L^2$-degree
$0$ reducibility, first defined in
\cite{jk}. \footnote{$L^2$-reducibility is the minimal sensible
  requirement in this context. That ``degree $0$'' is not a strong restriction
 is illustrated by the fact that already
$C^0$-reducibility to a constant rotation implies  $L^2$-degree
$0$ reducibility, see Remark \ref{rem22}.}
Thus for the main theorem, we will use ``reducible'' in this weak sense, meaning
$L^2$-degree $0$ reducible, as defined in (\ref{degree0}). 
With the normalized transfer matrix cocycle $(\alpha, \tilde{A}_{|c|,E})$ defined in (\ref{normalA}),
we have
\begin{thm}\label{main}
Suppose for $c(\cdot)\in
C^\omega_{\frac{h}{2\pi}}(\T,\C\backslash\{0\})$, $\beta(\alpha)<h$,
the normalized cocycles $(\alpha, \tilde{A}_{|c|,E})$ are
 reducible for a.e. $E$ with respect to the density of states measure. Then for a.e. $x\in\T$
\begin{itemize}
\item if $w(c)= 0$, the spectra  of $\tilde{H}_c(x)$ are pure point.
\item if $w(c)\neq 0$, the spectra of  $\tilde{H}_c(x)$ are purely absolutely continuous.
\end{itemize}


\end{thm}
%



As an important application, we 
obtain the sharp arithmetic phase transition result for the extended Harper's model (see (\ref{defehm})) in the positive Lyapunov exponent region (see Theorem (\ref{transition})).

The extended Harper's model (EHM) acting on $l^2(\Z)$ is defined as follows:
\begin{align}\label{defehm}
(H_{\lambda, \alpha,\theta}u)_n=c(\theta+n\alpha)u_{n+1}+\tilde{c}(\theta+(n-1)\alpha)u_{n-1}+2\cos{2\pi(\theta+n\alpha)}u_{n}.
\end{align}
where $c(\theta)=\lambda_1 e^{-2\pi i(\theta+\frac{\alpha}{2})}+\lambda_2+\lambda_3 e^{2\pi i(\theta+\frac{\alpha}{2})}$.
It was first proposed by D.J. Thouless in 1983 \cite{thou} and arises when 2D electrons are allowed to hop to both nearest neighboring  (expressed through $\lambda_2$) and the next-nearest lattice sites (expressed through $\lambda_1$ and $\lambda_3$).
This model includes almost Mathieu operator as a special case (when $\lambda_1=\lambda_3=0$). It is the central (non-Schr\"odinger) Jacobi operator, and has been a subject of many investigations, especially in physics literature. Important advances in the last decade include \cite{JKS,Jm,rui}. 
Recently, spectral decomposition for all $\lambda$ and a.e.  $\alpha,\theta$ was established in \cite{ajm}, thus making arithmetic spectral transitions in frequency the next natural object of study for this model. 

While in the subcritical regime (characterized by
the Lyapunov exponent vanishing in the complexified strip) it is
expected that, as in the corresponding almost Mathieu regime \cite{artur}, the
spectral properties do not depend on $\alpha$, in
the supercritical regime (that of positive Lyapunov exponents) there is
definitely a different behavior for the Diophantine \cite{JKS} and
Liouville \cite{kos} $\alpha$ thus it is interesting to determine a transition.
In the last few years, there have been several remarkable
developments, where in models with classical small denominator
problems leading to arithmetic transitions in spectral behavior, sharp
results were obtained, with analysis up to the very arithmetic
threshold. For the Maryland model, the spectral phase diagram was
determined {\it exactly} for all $\alpha,\theta$ in \cite{mar}. For
the almost Mathieu operator, the transition in $\alpha$ (conjectured
in 1994 \cite{conj}) was recently proved  in \cite{ayz}.
 Even more recently, pure point spectrum up to the transition was
 established by a different method in \cite{jl, jl2} with also an
 arithmetic condition on $\theta.$ Our main application adds to this
 growing collection by establishing a sharp transition in $\alpha$ for
 the extended Harper's operator.

An important feature of the extended Harper's model is that Lyapunov
exponents when restricted to the spectrum are constant and depend only
on $\lambda$ \cite{Jm}. Let $L(\lambda)$ (see (\ref{ehmLyapunov})) be the Lyapunov exponent
of $H_{\lambda, \alpha,\theta}$ for $E$ in the spectrum. It has been
computed exactly as a function of $\lambda$ \cite{thou,mz,Jm}.
 The exponential decay
provided by $L(\lambda)>0$ is in direct competition with small
denominators, with their smallness determined by the rate of growth of
continued fraction approximants of $\alpha$  and quantified by a
parameter $\beta(\alpha)$ defined in (\ref{beta})).  Resolving this competition
in a sharp way,  we have
\begin{thm}\label{transition} When $L(\lambda)>0$,
\begin{itemize}
\item if $\beta(\alpha)<L(\lambda)$, $H_{\lambda, \alpha, \theta}$ has pure point spectrum for a.e. $\theta$.
\item if $L(\lambda)<\beta(\alpha)$, $H_{\lambda,\alpha,\theta}$ has purely singular continuous spectrum for a.e. $\theta$.
\end{itemize}
\end{thm}
\begin{rem}
$L(\lambda)>0$ if and only if $1>\max{(\lambda_1+\lambda_3, \lambda_2)}$, see Theorem \ref{LE}.
\end{rem}

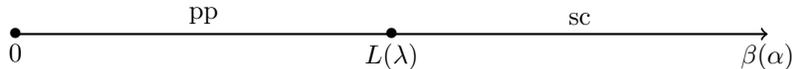
\begin{center}
\begin{tikzpicture}[thick, scale=1]

    \draw[->] (-10,-1) -- (0,-1) node[below] {$\beta(\alpha)$};

    \draw(-2.5,-1) node [above] {sc};
    \draw(-7.5,-1) node [above] {pp};
    \draw(-10, -1) node [below] {$0$};
    \draw(-5,   -1) node [below] {$L(\lambda)$};
    \draw(-5,   -1) node [] {$\bullet$};
    \draw(-10, -1) node [] {$\bullet$};
\end{tikzpicture}
\end{center}

The second statement of Theorem \ref{transition} does not require a
specific form of $c(\theta)$ and holds for general analytic $c$ that
are even allowed to vanish on $\T$. Namely, one can define a
coefficient $\delta_c(\alpha,\theta)\in[-\infty,\infty],$ dependent on
$c(\theta)$ through its zeros on $\T$ only, see (\ref{defdeltac}), and
satisfying $\delta_c(\alpha,\theta) =\beta(\alpha)$ for a.e. $\theta,$
so that

\begin{thm}\label{sing}
For any Lipshitz $v$, $ H_{c}( \theta) $ 
has no eigenvalues on $\{ E:L(E)< \delta_c(\alpha, \theta)\}$. In
particular, for a.e. $\theta$ it has no eigenvalues on $\{ E:L(E)< \beta(\alpha)\}$.
\end{thm}

This immediately implies
 
\begin{cor}\label{cor}
 If $L(E)>0$ for a.e. $E$ (in particular, if there exists
 $\theta_0\in\T$ with $c(\theta_0)=0$), then $ H_{\alpha, \theta} $ has purely singular continuous spectrum on
$\{ E:L(E)< \delta_c(\alpha, \theta)\}$. If $L(E)>0$ for a.e. $E$ and
$c$ doesn't vanish on $\T$, then $ H_{\alpha, \theta} $ has purely
singular continuous spectrum on $\{ E:L(E)< \beta(\alpha)\}$, for all $\theta.$
\end{cor}

We prove Theorem \ref{main} in Section \ref{mainsection}. We first
show that zero $w(c)$ ensures that elements of the reducibility matrix
can be used to construct eigenfunctions for the dual model. Then we
employ an argument of \cite{jk} to show completeness of
  those eigenfunctions. To prove the second part we establish a
  unitary conjugacy in $L^2(\T\times \Z)$ between $\tilde{H}_c$ and
  $\tilde{H}_s$  for a certain $s$ with $w(s)=0,$ ensuring that, by part one, $\tilde{H}_s(x)$ has
  pure point spectrum for a.e. $x.$ We then use those eigenfunctions to construct a large family of
  vector-valued functions $\psi^{q,\ell,j}_x(\cdot)$ such that for a.e. $x$
  the corresponding spectral measures are constant in $x.$ Finally, we prove their
  absolute continuity based again on the reducibility for the original
  model and argue completeness. Once Theorem \ref{main} is proved, 
to establish  the pure point part of Theorem \ref{transition} all we
need is the dual reducibility which follows quickly from a combination
of \cite{arc,afk}, similarly to the argument of \cite{ayz}. This is done in Section \ref{ppsection}. In fact, Theorem
\ref{transition} is an extension of the main theorem of \cite{ayz},
and specializes to it when $\lambda_1=\lambda_3=0.$  

The singular
continuous part as well as the general
Theorem \ref{sing} are proved in Section \ref{singsection}. The result is similar in
spirit to the recent theorems on meromorphic potentials
\cite{mar,jy}. The non-singular case is simpler and could follow
similarly to the singular continuous part of \cite{ayz} but we choose
to treat it together with the more involved singular case. While the Jacobi situation is quite different, the common
feature of singular Jacobi and meromorphic cocycles is their
singularity, leading to certain shared phenomena.  It is an interesting
question whether the first statement of the Corollary \ref{cor} is sharp at least in some situations,
so whether like in the Maryland model \cite{mar}, there is pure point spectrum for the complementary
set of $\theta.$ It is also interesting to see whether the second
statement is sharp for general analytic potentials (something still
far from reach even in the Schr\"odinger case).

\section{Preliminaries}
For a bounded analytic function $f$ defined on a strip $\{|\mathrm{Im}\theta|<\epsilon\}$ we let $\|f\|_{\epsilon}= \sup_{|\mathrm{Im}\theta|<\epsilon}|f(\theta)|$. 
If $f$ is a bounded continuous function on $\R$, we let $\|f\|_{0} = \sup_{\theta\in \R} |f(\theta)|$. 
For a set $U\subset \R$ let $|U|$ be the Lebesgue measure of $U$.

\subsection{Rational approximation}
For $\alpha\in \R\setminus \Q$, let $\beta(\alpha)\in [0, \infty)$ be given by
\begin{align}\label{beta}
\beta(\alpha)=\limsup_{n\rightarrow \infty} \frac{\ln{q_{n+1}}}{q_n},
\end{align}
where $\{\frac{p_n}{q_n}\}$ is the continued fraction approximants of $\alpha$. $\beta(\alpha)$ being large means $\alpha$ can be approximated very well by rational numbers. The following properties are well known
\begin{align}\label{qnqn+1}
\frac{1}{q_{n+1}}\leq \|q_n\alpha\|_{\T}\leq \frac{2}{q_{n+1}},
\end{align}
\begin{align}\label{qnbest}
\|q_n\alpha\|_{\T}\leq \|k\alpha\|_{\T}\ \mathrm{for}\ \mathrm{any}\ 1\leq |k|\leq q_{n+1}-1,
\end{align}
where $\|x\|_{\T}=\dist{(x, \Z)}$ for $x\in \R$.

\subsection{Winding number}
For $c(\cdot)\in C^{\omega}(\T, \C\backslash \{0\})$ on $\T$, let
\begin{align}\label{winding}
w(c)=\int_{\T}\frac{c^{\prime}(\theta)}{c(\theta)}\ \mathrm{d}\theta
\end{align}
be the winding number of $c(\theta)$. It describes how many times does the graph of $c(\theta)$ circle around the origin when $\theta$ goes along $\T$.

\subsection{Cocycles and Lyapunov exponent}
Let $\alpha\in \R\setminus \Q$ and $A\in C^0(\T, M_2(\C))$ with $\log{\|A(\cdot)\|}\in L^1(\T)$. The quasi-periodic cocycle $(\alpha, A)$ is the dynamical system on $\T\times \C^2$ defined by $(\alpha, A)(x, v)=(x+\alpha, A(x)v)$. The {Lyapunov exponent} is defined by the formula
\begin{align}\label{defLE}
L(\alpha, A)=\lim_{n\rightarrow\infty}\frac{1}{n}\int_{\T}\log{\|A_n(x)\|}\mathrm{d}x=\inf_{n}\frac{1}{n}\int_{\T}\log{\|A_n(x)\|}\mathrm{d}x.
\end{align}
where
\begin{align*}
A_n(x)= A(x+(n-1)\alpha)\cdots A(x) \ \ \mathrm{for}\ n\geq 0.
\end{align*}

\begin{lemma}\label{uniformupp}$\mathrm{(}$e.g.\cite{aj2}$\mathrm{)}$
Let $(\alpha, A)$ be a continuous cocycle, then for any $\delta>0$ there exists $C_{\delta}>0$ such that for any $n\in \N$ and $\theta\in \T$ we have
\begin{align*}
\|A_n(\theta)\|\leq C_{\delta} e^{(L(\alpha, A)+\delta)n}.
\end{align*}
\end{lemma}
\begin{rem}\label{intcontrolprod}
If we apply the previous lemma to one dimensional cocycle, we have that for any continuous function $z$, if $\ln{|z(\theta)|}\in L^1(\T)$ then for any $\epsilon>0$ there exists constant $C>0$ so that for any $a\leq b\in \Z$.
\begin{align*}
\prod_{k=a}^{b}|z(\theta+k\alpha)|\leq Ce^{(b-a+1)(\int_{\T}\ln{|z(\theta)|}d\theta+\epsilon)}\ \mathrm{for}\ \mathrm{any}\ \theta\in \T.
\end{align*}
\end{rem}

\subsection{Reducibility and rotation number}
Let 
\begin{align*}
R_{\theta}=
\left(
\begin{matrix}
\cos{2\pi \theta} \ \ &-\sin{2\pi \theta}\\
\sin{2\pi \theta} \ \ &\cos{2\pi \theta}
\end{matrix}
\right).
\end{align*}
Any $A\in C^0(\T, PSL(2,\R))$ is homotopic to $\theta \rightarrow R_{\frac{n}{2}\theta}$ for some $n\in \Z$, called the {\it degree} of $A$, denoted by $\deg{A}=n$.

Assume now that $A\in C^0(\T, SL(2,\R))$ is homotopic to identity.
Then there exists $\psi:\R/\Z \times \R/\Z \to \R$ and $u:\R/\Z \times \R/\Z \rightarrow \R^+$ such that
\begin{align*}
A(x) \cdot 
\left(
\begin{matrix}
\cos{2\pi y} \\ \sin{2\pi y} 
\end{matrix} 
\right)
=u(x,y)
\left(
\begin{matrix} 
\cos{2 \pi (y+\psi(x,y))} \\ \sin{2 \pi (y+\psi(x,y))} 
\end{matrix} 
\right).
\end{align*}
The function $\psi$ is called a {\it lift} of $A$.  
Let $\mu$ be any probability on $\R/\Z \times \R/\Z$ which is invariant by the continuous
map $T:(x,y) \mapsto (x+\alpha,y+\psi(x,y))$, projecting over Lebesgue
measure on the first coordinate (for instance, take $\mu$ as any
accumulation point of $\frac {1} {n} \sum_{k=0}^{n-1} T_*^k \nu$ where
$\nu$ is Lebesgue measure on $\R/\Z \times \R/\Z$).  
Then the number
\begin{align}\label{defrot}
\rho(\alpha,A)=\int \psi d\mu \mod \Z
\end{align}
does not depend on the choices of $\psi$ and $\mu$, and is called the
{ fibered rotation number} of
$(\alpha,A)$.

The fibered rotation number is invariant under real conjugacies which are
homotopic to the identity.  In general, if $(\alpha,A^{(1)})$ and
$(\alpha,A^{(2)})$ are real conjugate,
namely there exists $B\in C^0(\T, PSL(2, \R))$ so that $B^{-1}(x+\alpha)A^{(2)}(x)B(x)=A^{(1)}(x)$ and $\deg{B}=k$, then
\begin{align}\label{rhorhorelation}
\rho(\alpha,A^{(1)})=\rho(\alpha,A^{(2)})-k\alpha/2.
\end{align}

We say a cocycle $(\alpha, A)$ is {\it $C^{\omega}$ rotation-reducible} if it is real analytically conjugate to a rotation matrix, namely there exists $B\in C^{\omega}(\T, PSL(2, \R))$ and a function $\psi$ on $\T$, such that
\begin{align*}
B^{-1}(\theta+\alpha)A(\theta)B(\theta)=R_{\psi(\theta)}.
\end{align*}

A cocycle $(\alpha, A)$ is called {\it ${L^2}$-reducible} (or {\it $C^0$-reducible, $C^\omega$-reducible}) if there exists a matrix function $B\in L^2(\T, SL(2, \R))$ (or $B\in C^0(\T, SL(2, \R))$, $B\in C^\omega(\T, PSL(2, \R))$ respectively) such that
\begin{align}\label{defreducible}
B^{-1}(\theta+\alpha)A(\theta)B(\theta)=A_{*}\ \ \mathrm{for}\ \mathrm{a.e.}\ \theta\in \T,
\end{align}
with $A_*$ being a constant matrix. 

A cocycle $(\alpha, A)$ is called $L^2$-degree 0 reducible if (\ref{defreducible}) holds with
\begin{align}\label{degree0}
A_*=R_{\rho(\alpha, A)}.
\end{align}
\begin{rem}\label{rem22}
Suppose $(\alpha, A)$ is a $C^0$-reducible such that (\ref{defreducible}) holds with $A_{*}$ being a constant rotation matrix. 
Then it is $L^2$ (and $C^0$) degree 0 reducible.
\end{rem}

\subsection{Jacobi operators and normalized cocycles}
Given a family of Jacobi operators $\{H_{c}(\theta)\}_{\theta}$, a formal solution to the equation $H_{c}(\theta)u=Eu$ can be reconstructed via the following equation:
\begin{align*}
\left(
\begin{matrix}
u_{n+1}\\
u_{n}
\end{matrix}
\right)
=
A_{c,E}(\theta+n\alpha)
\left(
\begin{matrix}
u_{n}\\
u_{n-1}
\end{matrix}
\right)
\end{align*} 
where the transfer matrix is
\begin{align*}
A_{c,E}(\theta)=\frac{1}{c(\theta)}D_{c,E}(\theta)=\frac{1}{c(\theta)}
\left(
\begin{matrix}
E-v(\theta)\ \ &-\tilde{c}(\theta-\alpha)\\
c(\theta)\ \ & 0
\end{matrix}
\right).
\end{align*}
This means the cocycle directly related to the Jacobi operator is $(\alpha, A_{c,E})$. 
However this cocycle does have some disadvantages: first of all if $c\not\equiv 1$, $A_{c,E}(\theta)\notin SL(2,\R)$, secondly when $w(c)\neq 0$ (see (\ref{winding})) $A_{c, E}(\theta)$ is not homotopic to identity. Thus one couldn't properly apply the techniques and results obtained by the reducibility methods directly to this cocycle. 
We will instead work with a normalized cocycle $(\alpha, \tilde{A}_{|c|,E})$, where
\begin{align}\label{normalA}
\tilde{A}_{|c|,E}(\theta)=
\frac{1}{\sqrt{|c|(\theta)|c|(\theta-\alpha)}}D_{|c|,E}(\theta)=
\frac{1}{\sqrt{|c|(\theta)|c|(\theta-\alpha)}}
\left(
\begin{matrix}
E-v(\theta)\ \ &-|c|(\theta-\alpha)\\
|c|(\theta)\ \ &0
\end{matrix}
\right).
\end{align}
The cocycle $(\alpha, \tilde{A}_{|c|,E})$ can be used effectively. First, $(\alpha, \tilde{A}_{|c|, E})$ is closely related to $(\alpha, A_{c, E})$ in the sense that there is an explicit conjugation relation between $D_{c,E}$ and $D_{|c|,E}$, given by
\begin{align}\label{relateDcDnor}
D_{c,E}(\theta)=M_{c}(\theta+\alpha)D_{|c|,E}(\theta)M^{-1}_{c}(\theta),
\end{align}
where 
\begin{align}\label{conjDcDnor}
M_c(\theta)=
\left(
\begin{matrix}
1\ \ &0\\
0 \ \ &\sqrt{\frac{c(\theta-\alpha)}{\tilde{c}(\theta-\alpha)}}
\end{matrix}
\right).
\end{align}
Furthermore,
$(\alpha, {\tilde{A}_{|c|,E}})$ is homotopic to the identity in $C^{\omega}(\T, SL(2,\R))$. Indeed,
\begin{align*}
\tilde{A}_{|c|,E}(\theta,t)=\frac{1}{\sqrt{|c|(\theta)|c|(\theta-t\alpha)}}
\left(
\begin{matrix}
t(E-v(\theta))\ \ &-|c|(\theta-t\alpha)\\
|c|(\theta)   \ \ &0
\end{matrix}
\right),
\end{align*}
establishes a homotopy of $(\alpha, {\tilde{A}_{|c|,E}})$ to the constant real rotation by $\frac{\pi}{2}$ and hence to the identity matrix.

Therefore we can define the rotation number of the cocycle $(\alpha, \tilde{A}_{|c|, E})$ by (\ref{defrot}), denoted by $\rho(\alpha, {\tilde{A}_{|c|,E}})$. It is a non-increasing continuous function of $E$. Clearly, $\rho(\alpha, {\tilde{A}_{|c|,E}})$ is the rotation number associated to the operator:
\begin{align}\label{H|c|}
(H_{|c|}(\theta)u)_n=|c|(\theta+n\alpha) u_{n+1}+|c|(\theta+(n-1)\alpha)u_{n-1}+v(\theta+n\alpha) u_n.
\end{align}
Sometimes we will write $\rho_{|c|}(E)$ instead of $\rho(\alpha, \tilde{A}_{|c|,E})$ for simplicity.

Let $\mu_{c,\theta}$ be the spectral measure of $H_{c}(\theta)$ corresponding to $\delta_0$, namely for any Borel set $U$ we have
\begin{align*}
\mu_{c,\theta}(U)=(\delta_0, \chi_U(H_{c}(\theta))\delta_0).
\end{align*}
We define the {\it density of states measure} $dN_c$ by
\begin{align*}
dN_c(U)=\int_{\T}\mu_{c,\theta}(U)\ \mathrm{d}\theta.
\end{align*}
$N_c(E):=N_c(-\infty, E)$ is called the {\it integrated density of states (IDS)} of $H_{c}(\theta)$. 
It can be seen that $N_{|c|}(E)=N_{c}(E)$ since $H_{|c|}(\theta)$ and $H_{c}(\theta)$ differ by a unitary conjugation. 
Indeed, let $\{T_{\theta}\}_{\theta\in \T}$ be a family of unitary operators acting on $l^2(\Z)$ as follows
\begin{align}\label{Ttheta}
(T_{\theta} u)_n=\left\lbrace 
\begin{matrix}
\sqrt{\frac{\prod_{j=0}^{n-1} \tilde{c}(\theta+j\alpha)}{\prod_{j=0}^{n-1} c(\theta+j\alpha)}}u_n\ \ &n\geq 1,\\
u_n\ \ &n=0,\\
\sqrt{\frac{\prod_{j=n}^{-1}c(\theta+j\alpha)}{\prod_{j=n}^{-1}\tilde{c}(\theta+j\alpha)}}u_n\ \ &n\leq -1,
\end{matrix}
\right.
\end{align}
then 
\begin{align}\label{HcH|c|}
H_{|c|}(\theta)=T_{\theta}^{-1}H_c(\theta)T_{\theta}.
\end{align}
Combining (\ref{HcH|c|}) with the observation that $T_{\theta}\delta_0=\delta_0$ and $(T_{\theta}^{-1})^*=T_{\theta}$, we have that, for any Borel set $U$,
\begin{align}\label{a}
N_{|c|}(U)=&\int_{\T}(\delta_0, \chi_{U}(H_{|c|}(\theta))\delta_0)\ d\theta\\
=&\int_{\T}(\delta_0, T_{\theta}^{-1}\chi_{U}(H_c(\theta))T_{\theta}\delta_0)\ d\theta\notag\\
=&\int_{\T}(T_{\theta}\delta_0, \chi_{U}(H_c(\theta)) T_{\theta}\delta_0)\ d\theta\notag\\
=&\int_{\T}(\delta_0, \chi_{U}(H_c(\theta))\delta_0)\ d\theta\notag\\
=&N_c(U)\notag.
\end{align}

For the operator $H_{|c|}(\theta)$, the connection between its rotation number and integrated density of states is the following
\begin{align}\label{Nrho}
N_{|c|}(E)=1-2\rho_{|c|}(E).
\end{align}

The Aubry duality between $H_c$ and $\tilde{H}_c$ implies the following relation between their density of states measures, see e.g. a particular case of Theorem 2 in \cite{mz},\footnote{It can also be proved
  essentially by the argument, as above, that proves equality of
  $N_{|c|}$ and $N_c,$ replacing (\ref{HcH|c|}) with Lemma
  \ref{aubry}, where $H_c$ and $\tilde{H}_c $ are operators acting on
  $L^2(\T\times \Z)$  defined as direct integrals in
$\theta$ of $H_c(\theta)$ and $\tilde{H}_c(\theta).$  While
  $H_{|c|}(\theta)$ and $H_c(\theta)$ are unitary conjugate in each
  fiber, this is not essential, for the proof only uses
  $T\delta_0=\delta_0$ and unitarity of the conjugation.}
\begin{align}\label{IDSrelation}
dN_{c}(E)=dN_{\tilde{H}_c}(E).
\end{align}

\subsection{Extended Harper's model}\label{ehmsec}

We consider the extended Harper's model $\{H_{\lambda,\alpha,\theta}\}_{\theta \in \T}$.  
The spectrum of $H_{\lambda,\alpha,\theta}$ denoted by
$\Sigma_{\lambda, \alpha}$, does not depend on $\theta$.
Depending on the values of the parameters $\lambda_1, \lambda_2, \lambda_3$, we could divide the parameter space into three regions as shown in the picture below:
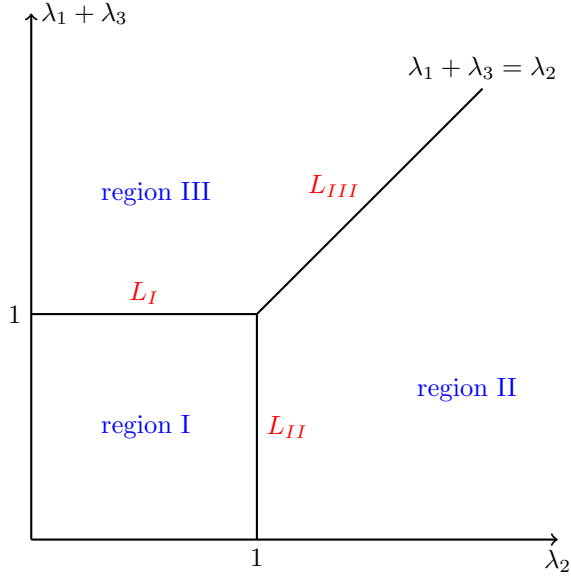
\begin{center}
\begin{tikzpicture}[thick, scale=1]

    \draw[->] (-10,-1) -- (-3,-1) node[below] {$\lambda_2$};
    \draw[->] (-10,-1) -- (-10,6) node[right] {$\lambda_1+\lambda_3$};
    \draw [ ] plot [smooth] coordinates { (- 7, 2) (-4, 5) };
    \draw [ ] plot [smooth] coordinates { (- 7, 2) (-7,-1) };
    \draw [ ] plot [smooth] coordinates { (-10, 2) (-7, 2) };

    \draw(-  4,   5) node [above] {$\lambda_1+\lambda_3=\lambda_2$};
    \draw(-  7,  -1) node [below] {$1$};
    \draw(- 10,   2) node [left]  {$1$};
    \draw(-9.2, 0.5) node [color=blue][right] {region I};
    \draw(-  5,   1) node [color=blue][right] {region II};
    \draw(-9.2, 3.6) node [color=blue][right] {region III};

    \draw(-  7, 0.5) node [color=red][right] {$L_{II}$};
    \draw(-8.5,   2) node [color=red][above] {$L_{I}$};
    \draw(-5.5, 3.7) node [color=red][left] {$L_{III}$};

\end{tikzpicture}
\end{center}

\begin{align*}
&region\ I    : 0<\max{(\lambda_1+\lambda_3,\lambda_2)}<1,\\
&region\ II   : 0<\max{(\lambda_1+\lambda_3, 1)} < \lambda_2,\\
&region\ III  : 0<\max{(1, \lambda_2)} < \lambda_1+\lambda_3.
\end{align*}
According to the action of the duality transformation $\sigma: \lambda=(\lambda_1, \lambda_2, \lambda_3)\rightarrow \hat{\lambda}=(\frac{\lambda_3}{\lambda_2}, \frac{1}{\lambda_2}, \frac{\lambda_1}{\lambda_2})$, region I and region II are dual to each other, and region III is  self-dual.
It turns out the spectrum $\Sigma_{\lambda, \alpha}$ of $H_{\lambda,\alpha, \theta}$ is related to the spectrum $\Sigma_{\hat{\lambda}, \alpha}$ of $H_{\hat{\lambda},\alpha,\theta}$ in the following way
\begin{align*}
\Sigma_{\lambda, \alpha}=\lambda_2 \Sigma_{\hat{\lambda}, \alpha}.
\end{align*}
By (\ref{IDSrelation}), the $\mathrm{IDS}$ $N(E)$ of $H_{\lambda,\alpha,\theta}$ coincides with the $\mathrm{IDS}$ $\hat{N}({E}/{\lambda_2})$ of $H_{\hat{\lambda},\alpha,\theta}$.

In \cite{Global}, Avila divides all the energies in the spectrum into three categories: super-critical, namely the energies with positive Lyapunov exponent; subcritical, namely the energies whose Lyapunov exponent of the phase-complexified cocycles are identically equal to zero in a neighborhood of $\epsilon=0$; critical, otherwise. Regarding EHM, the following theorem is shown in \cite{Jm}:
\begin{thm}\label{LE}\cite{Jm}
The extended Harper's model is {\it super-critical} in region I and {\it sub-critical} in region II. Indeed, if $\lambda$ belongs to region I,
\begin{itemize} 
\item for any $E\in \Sigma_{\lambda, \alpha}$, 
\begin{align}\label{ehmLyapunov}
L(E)\equiv L(\lambda)=\ln{\frac{1+\sqrt{1-4\lambda_1\lambda_3}}{\max{(\lambda_1+\lambda_3,\lambda_2)}+\sqrt{\max{(\lambda_1+\lambda_3,\lambda_2)}^2-4\lambda_1\lambda_3}}}>0.
\end{align}
\item $\hat{\lambda}$ belongs to region II, furthermore, for any $E\in \Sigma_{\hat{\lambda}, \alpha}$,  
\begin{align}\label{ehmregion2=0}
L(E)=L(\alpha, A_{d, E}(\cdot+i\epsilon))=L(\alpha, \tilde{A}_{|d|, E}(\cdot+i\epsilon))=0\ \ \mathrm{for}\ \ |\epsilon|\leq \frac{L(\lambda)}{2\pi},
\end{align} 
where $(\alpha, A_{d, E})$ is the directly related cocycle to $H_{\hat{\lambda}, \alpha, \theta}$ and $d(\theta)=\frac{\lambda_3}{\lambda_2}e^{-2\pi i (\theta+\frac{\alpha}{2})}+\frac{1}{\lambda_2}+\frac{\lambda_1}{\lambda_2}e^{2\pi i (\theta+\frac{\alpha}{2})}$.
\end{itemize}
\end{thm}

\section{Proof of Theorem \ref{main}}\label{mainsection}
The proof of Theorem $\ref{main}$ mainly relies on Lemmas \ref{winding=0pp} and \ref{windingneq0}.
\subsection{Proof of the first part of Theorem \ref{main}}
\begin{lemma}\label{winding=0pp}
Let $s(\theta)=c(\theta)e^{-2\pi i k_0 (\theta+\frac{\alpha}{2})}$, where $k_0=w(c)$. Then under the conditions of Theorem $\ref{main}$, the spectra of the dual Hamiltonians $\tilde{H}_{s}(x)$ are pure point for a.e. $x$.
\end{lemma}

\begin{proof}
We start with
\begin{lemma}\label{analytic}
Suppose $s(\theta)=\sum_{k\in \Z} \hat{s}_k e^{2\pi i k (\theta+\frac{\alpha}{2})}$, $\hat{s}_k\in \R$.
Suppose $s(\cdot)$ is analytic and nonzero on $|\mathrm{Im}\theta|\leq \frac{h}{2\pi}$ and $w(s)=0$. 
Then if $\beta(\alpha)<h$ there exists analytic funtion $f(\theta)$ such that
\begin{align*}
\frac{s(\theta)}{|s|(\theta)}=e^{f(\theta+\alpha)-f(\theta)}.
\end{align*}
\end{lemma}

\begin{proof}
Since $w(s(\cdot+i\epsilon))\equiv 0$, we can properly define $\log{s}(\theta)$ and $\arg{s(\theta)}$ on $|\mathrm{Im}\theta|\leq \frac{h}{2\pi}$. 
Now that obviously $\tilde{s}(\theta)=s(-\theta-\alpha)$, we have
\begin{align}\label{analytic1}
\int_{\T} \ln{|s(\theta)|}\ \mathrm{d}\theta=\int_{\T} \ln{|\tilde{s}(\theta)|}\ \mathrm{d}\theta.
\end{align}
and
\begin{align}\label{analytic2}
\int_{\T} \arg{s(\theta)}\ \mathrm{d}\theta-\int_{\T} \arg{\tilde{s}(\theta)}\ \mathrm{d}\theta=\int_{\T} \arg{s(\theta)}\ \mathrm{d}\theta-\int_{\T} \arg{s(-\theta-\alpha)}\ \mathrm{d}\theta=0.
\end{align}
Combining (\ref{analytic1}), (\ref{analytic2}) with $\beta(\alpha)<h$ we are able to solve a coholomogical equation, hence there exists an analytic function $g(\theta)$ so that
\begin{align*}
g(\theta+\alpha)-g(\theta)=\ln{s(\theta)}-\ln{\tilde{s}(\theta)}.
\end{align*}
This clearly implies
\begin{align*}
\frac{s(\theta)}{\tilde{s}(\theta)}=e^{g(\theta+\alpha)-g(\theta)}.
\end{align*}
Hence
\begin{align*}
\frac{s(\theta)}{|s|(\theta)}=e^{f(\theta+\alpha)-f(\theta)},
\end{align*}
where $f(\theta)=\frac{1}{2}g(\theta)$. $\hfill{} \Box$
\end{proof}

Now let us come back to the proof of Lemma \ref{winding=0pp}. 
We have for a.e. $E$ with respect to the density of states measure
$dN_c$ \footnote{It is the same as $dN_{|c|}=dN_{|s|}$, by
  (\ref{a}) and since
  $|c|=|s|$ on $\T^1.$} 
, there is $B_E\in L^2(\T, SL(2,\R))$ so that
\begin{align}\label{L2reduciblePc}
B_E^{-1}(\theta+\alpha)\tilde{A}_{|c|,E}(\theta)B_E(\theta)=R_{\rho_{|c|}(E)}.
\end{align}
Since for $\theta\in \T$, $\tilde{A}_{|c|, E}(\theta)=\tilde{A}_{|s|, E}(\theta)$, we have
\begin{align}\label{L2reduciblePs}
B_E^{-1}(\theta+\alpha)\tilde{A}_{|s|,E}(\theta)B_E(\theta)=R_{\rho_{|s|}(E)}.
\end{align}
Taking 
\begin{align*}
\tilde{B}_E(\theta)=\frac{1}{\sqrt{2i}}B_E(\theta)\left(\begin{matrix}
i\ \ &-i\\
1\ \ &1
\end{matrix}\right),
\end{align*}
we have,
\begin{align}\label{reducetorot}
\tilde{B}_E^{-1}(\theta+\alpha)
\tilde{A}_{|s|, E}(\theta)
\tilde{B}_E(\theta)=
\left(\begin{matrix}
e^{2\pi i \rho_{|s|}(E)}\ \ &0\\
0\ \ &e^{-2\pi i \rho_{|s|}(E)}
\end{matrix}\right).
\end{align}
By Lemma \ref{analytic}, there exists analytic $f(\theta)$ so that
$s(\theta)=|s|(\theta)e^{f(\theta+\alpha)-f(\theta)}$.
Then by (\ref{relateDcDnor}) and (\ref{reducetorot}), we have
\begin{align}\label{Asreducerot}
\left(
\begin{matrix}
e^{2\pi i \rho_{|s|}(E)}\ \ &0\\
0 \ \ &e^{-2\pi i \rho_{|s|}(E)}
\end{matrix}
\right)=&B^{-1}_E(\theta+\alpha)\tilde{A}_{|s|,E}(\theta)B_E(\theta)\\
=
&\frac{s(\theta)}{\sqrt{|s|(\theta)|s|(\theta-\alpha)}}\lbrace M_{s}(\theta+\alpha)B_E(\theta+\alpha)\rbrace^{-1} A_{s,E}(\theta)
M_s(\theta)B_E(\theta) \notag\\
=
&\left\lbrace 
\frac{M_s(\theta+\alpha)B_E(\theta+\alpha)e^{-\frac{f(\theta+\alpha)}{2}}}{\sqrt{|s|(\theta)}}
\right\rbrace ^{-1} 
A_{s,E}(\theta)
\left\lbrace
\frac{M_s(\theta)B_E(\theta)e^{-\frac{f(\theta)}{2}}}{\sqrt{|s|(\theta-\alpha)}} 
\right\rbrace. \notag
\end{align}
Let 
$
\tilde{D}_E(\theta)=
\left(
\begin{matrix}
D_{E,11}(\theta)  &D_{E,12}(\theta)\\
D_{E,21}(\theta)  &D_{E,22}(\theta)
\end{matrix}
\right):=\lbrace
\frac{M_s(\theta)B_E(\theta)e^{-\frac{f(\theta)}{2}}}{\sqrt{|s|(\theta-\alpha)}} 
\rbrace
$.
(\ref{Asreducerot}) yields that
\begin{align}\label{B}
(E-v(\theta))D_{E,11}(\theta)&=e^{2\pi i \rho_{|s|}(E)} s(\theta) D_{E,11}(\theta+\alpha)+e^{-2\pi i \rho_{|s|}(E)} \tilde{s}(\theta-\alpha) D_{E,11}(\theta-\alpha),\\
(E-v(\theta))D_{E,21}(\theta)&=e^{-2\pi i \rho_{|s|}(E)} s(\theta)D_{E,21}(\theta+\alpha)+e^{ 2\pi i \rho_{|s|}(E)}\tilde{s}(\theta-\alpha)  D_{E,21}(\theta-\alpha).
\end{align}
We now can follow the argument of \cite{jk}. We are going to show that
\begin{lemma}\label{complete}
For a.e. $x$, $\tilde{H}_s(x)$ has a complete set of normalized eigenfunctions with simple eigenvalues.
\end{lemma}
\begin{proof}
As mentioned, this proof is essentially from \cite{jk}, we include it here for completeness.
Since $\rho_{|s|}: \R\rightarrow [0,\frac{1}{2}]$ is bijective on the spectrum, for each $x\in [0,\frac{1}{2}]$ there exists $E(x)$ such that $\rho_{|s|}(E(x))=x$. 
By (\ref{B}) and a straightforward computation, there is $F_1$ with $|F_1|=0$ so that for $x\in [0,\frac{1}{2}]\setminus F_1$, $\tilde{H}_s(x)$ has a normalized eigenfunction 
$\{u_k(x)\}_k=\left\lbrace\frac{\hat{D}_{E(x),11}(k)}{\|\hat{D}_{E(x),11}\|_{L^2(\T)}}\right\rbrace_k$ at energy $E(x)$.
Also for $x\in [-\frac{1}{2}, 0]\setminus F_2$, $|F_2|=0$, $\tilde{H}_s(x)$ has a normalized eigenfunction 
$\{u_k(x)\}_k=\left\lbrace\frac{\hat{D}_{E(x),12}(k)}{\|\hat{D}_{E(x),12}\|_{L^2(\T)}}\right\rbrace_k$ at energy $E(-x)$.
Let 
\begin{align}\label{F}
F=(F_1+\Z\alpha) \cup (F_2+\Z\alpha) \cup \{x\in [-\frac{1}{2},\frac{1}{2}]\ |\ 2x\in \Z\alpha+\Z\}.
\end{align} 
Clearly, $|F|=0$.
Now for every $x\in F^c$, every $n\in \Z$, $\tilde{H}_s(x+n\alpha)$ has a normalized eigenfunction $\{u_k(x+n\alpha)\}_k$ at energy $E(x+n\alpha)$. 
Also for different $m$ and $n$, $E(x+m\alpha)\neq E(x+n\alpha)$, since otherwise we would have $x+m\alpha=-(x+n\alpha)\ \mod\Z$, which is impossible due to our definition of $F$, (\ref{F}).
Let $E_n(x):=E(x+n\alpha)$, $P_n(x)$ be the spectral projection of $\tilde{H}_s(x)$ onto $E_n(x)$ and $P(x)=\sum_{n\in \Z}P_n(x)$.
Notice that $\tilde{H}_s(x+n\alpha)=T^{-n}\tilde{H}_s(x)T^n$, where $(Tu)_k=u_{k-1}$. Thus
\begin{align*}
P_n(x) T^nu(x+n\alpha)=T^n u(x+n\alpha),
\end{align*}
in other words, $T^n u(x)$ is in the range of $P_n(x-n\alpha)$. 
Thus for any $l\in \Z$, $\langle \delta_l, P_n(x-n\alpha)\delta_l\rangle \geq |\langle\delta_l, T^n u(x)\rangle|^2$, therefore $\sum_{n\in \Z} \langle\delta_l, P_n(x-n\alpha)\delta_l\rangle \geq 1$.
We have
\begin{align*}
1\geq \int_{F^c}\langle \delta_l, P(x)\delta_l \rangle=\int_{F^c} \sum_{n\in \Z} \langle\delta_l, P_n(x-n\alpha)\delta_l\rangle \geq 1.
\end{align*}
This implies for a.e. $x$, $\langle \delta_l, P(x)\delta_l \rangle=1$ for every $l$, therefore $P(x)=1$.
Thus for a.e. $x\in\T$, $\cup_{n\in \Z} T^n u(x+n\alpha)$ forms a complete set of eigenfunctions and $\cup_{n\in \Z} E(x+n\alpha)$ forms the eigenvalues. 
$\hfill{} \Box$
\end{proof}

Note that this immediately implies Lemma \ref{winding=0pp}, and thus
also the first part of Theorem \ref{main}, since when $w(c)=0$, we have $\tilde{H}_s=\tilde{H}_c$. $\hfill{} \Box$
\end{proof}

As a byproduct of full measure $L^2$-degree 0 reducibility (\ref{L2reduciblePs}), we could obtain the following result about the absolute continuity of the density of states measure which will play an important role in the proof of Lemma \ref{windingneq0}.
\begin{lemma}\label{IDSac}
The density of states measure of $H_c$ (and thus of $H_s$, $\tilde{H}_c$ and $\tilde{H}_s$) is absolutely continuous.
\end{lemma}
\begin{proof}
By \cite{afk} (see Lemma 1.4 therein), if $(\alpha, \tilde{A}_{|s|, E})$ is $L^2$-degree 0 reducible for a.e. $E$ with respect to the density of states measure, then $(\alpha, \tilde{A}_{|s|, E})$ is $C^\omega$ rotation-reducible for $E\in U$ where $U$ is a set with $dN_{|s|}(U)=1$. 
By subordinacy theory \cite{subordinate}, $(\alpha, \tilde{A}_{|c|, E})$ being $C^{\omega}$ rotation-reducible for $E\in U$ implies that for any $\theta$, the singular part of the spectral measure $\mu_{|c|, \theta}$ of $H_{|c|}(\theta)$ gives zero weight to $U$. 
This implies, by Footnote 1, $dN_{|c|}=dN_{c}=dN_{s}$ are absolutely continuous. 
By (\ref{IDSrelation}), we get that $dN_{\tilde{H}_c}=dN_{\tilde{H}_s}$, the density of states measures of the dual Hamiltonians $\tilde{H}_c$ and $\tilde{H}_s$, are absolutely continuous.  $\hfill{} \Box$
\end{proof}

\subsection{Proof of the second part of Theorem \ref{main}}
\begin{lemma}\label{windingneq0}
If $w(c)=k_0\neq 0$, the spectra of $\tilde{H}_c$ are purely absolutely continuous for a.e. $x$.
\end{lemma}
\subsection*{Proof of Lemma \ref{windingneq0}}

The plan of the proof is to find a unitary transformation of
$L^2(\T\times \Z)$ relating $\tilde{H}_c$ to $\tilde{H}_s,$ and prove that the (already
established in the first part) a.e. pure point spectrum of $\tilde{H}_s(x)$ for a.e. $x$ leads to
absolutely continuous spectrum of $\tilde{H}_c(x)$ for a.e. $x.$

 Let us introduce two unitary transformations on $\mathcal{H}=L^2(\T\times \Z)$,
\begin{align}
(U_R\psi)(x, n)&=\int_{0}^{1}e^{2\pi i \beta n}\sum_{p\in \Z}\psi(\beta, p)e^{2\pi i p(x+n\alpha)}\ d\beta. \label{UR}\\
(U_{k}\psi)(x, n)&=e^{2\pi i (nk(\frac{n\alpha}{2}+x))}\psi(x, n).\label{UGk}
\end{align}
$U_R$, first introduced in \cite{cd},  is just the Aubry duality transformation, also given in a more
compact form as 
\begin{align} U_R\psi(x,n)=\hat{\psi}(n,x+\alpha n), \end{align}
  where $\hat{\psi}\in L^2(\Z\times \T)$ is the Fourier transform. Operator $U_k$, first introduced
in \cite{mz},  is unitary on each
fiber. We also have
\begin{align}\label{UR-1}
(U_{R}^{-1}\psi)(x, n)&=\int_{0}^{1}e^{-2\pi i \beta n}\sum_{p\in \Z}\psi(\beta, p)e^{-2\pi i p(x+n\alpha)}\ d\beta.
\end{align}

1D one-frequency quasiperiodic operators arise as reductions of 2D Hamiltonians in a
uniform magnetic field with the spectral properties of the 2D model
encoded in the properties of the entire phase-dependent family of
$H(\theta)$. 
Both $U_R$ and $U_k$arise from the  gauge transformations of the
original 2D model.
Namely, the $U_R$ transformation can be interpreted as a gauge transformation
corresponding to rotating the original 2D lattice $\Z_x\times \Z_y$ by the angle $\pi /2$, 
and the $U_k$ transformation can be interpreted as transformation
corresponding to shifting each 1D sub-lattice $\Z_x$ horizontally by $ky$ units.
One could refer to \cite{mz} for more details.

Now let us define $(S_m\psi)(x, n)=\psi(x+m\alpha, n-m)$. Then $(S_mv_{l, j})(x, n)=v_{l, j}(x+m\alpha, n-m)=v_{l, m+j}(x, n)$.
\begin{lemma}\label{URR-kS}
The following hold
\begin{align}
(U_RS_l\psi)(x, n)=&e^{2\pi i lx}(U_R\psi)(x, n) \label{URS}\\
(U_{R}^{-1}S_l\psi)(x, n)=&e^{-2\pi i lx}(U_{R}^{-1}\psi)(x, n) \label{UR-S}\\
(U_{k}S_l\psi)(x, n)=&e^{2\pi i lk(x+\frac{l\alpha}{2})}(S_lU_{k}\psi)(x, n). \label{UkS}
\end{align}
\end{lemma}
\begin{proof}
Straightforward computation. $\hfill{}\Box$
\end{proof}

Define operators $H_c,$ $\tilde{H}_c$ 
as acting on $\mathcal{H}$ via direct integrals in
$x$ of $H_c(x)$ and $\tilde{H}_c(x).$ Then one way to formulate the
Aubry duality is

\begin{lemma}\label{aubry}
\begin{align}\label{aaubry}
\tilde{H}_c=U_{R}^{-1}H_c U_{R}.
\end{align}
\end{lemma}
\begin{proof}
A computation using Lemma \ref{URR-kS}. $\hfill{}\Box$
\end{proof}

 Now we establish a connection between $\tilde{H}_s$ and
 $\tilde{H}_c.$ It is given by the following 
\begin{lemma}\label{AA}
\begin{align}\label{conjHsHc}
\tilde{H}_c=(U_{R}^{-1}U_{k_0}U_R)\tilde{H}_s (U_{R}^{-1}U_{k_0}U_R)^{-1}.
\end{align}
\end{lemma}
\begin{proof}
A more involved computation using Lemma \ref{URR-kS}. $\hfill{}\Box$
\end{proof}

By Lemma \ref{complete}, for $x\in F^c$ with $|F|=0$, $\tilde{H}_s(x)$ has a complete set of normalized eigenfunctions with simple eigenvalues.
First, we are going, following \cite{gjls}, to prove there is a
covariant measurable enumeration of this set.

For any $x\in F^c$, let $u(x, \cdot)$ be one of its normalized
eigenfunctions. Define $j(u(x))$ be the leftmost maximum for $|u(x,
\cdot)|$. We fix $u(x, \cdot)$ by requiring $u(x, j)>0$ and say it is
attached to $j$. The key observation is that the argument of Section 2
of \cite{gjls}, while formulated there for discrete one dimensional
Schr\"odinger operators, works verbatim for any dicrete
one-dimensional operator with simple eigenvalues. \footnote{The
  existence of measurable enumeration of eigenfunctions was proved, in
  great generality in \cite{gk}. However, since we need a covariant
  representation satisfying (\ref{ortho2}) the argument of \cite{gjls}
  is better suited to our needs}. Thus we get for a.e. $x$ a complete set of eigenfunctions $\{v_{l, j}(x, \cdot)\}_{l, j}$ with eigenvalues $\{e_{l, j}(x)\}$ so that
\begin{enumerate}
\item for each fixed $l, j$, $v_{l, j}(x, \cdot)$ and $e_{l, j}(x)$ are measurable functions of $x$.
\item $\{v_{l, j}(x, \cdot)\}_{j}$ are attached to $j$.
\item $v_{l, j}(x, j)\geq v_{l+1, j}(x, j)$. If the equality holds then $e_{l, j}(x)>e_{l+1, j}(x)$. \footnote{For fixed $l, j$, generally, $v_{l, j}(x)$ may vanish identically on a positive measure set of $x$}
\end{enumerate}
\

By simplicity of the eigenvalues, for any $(l, j)\neq (l^\prime, j^\prime)$ we have
\begin{align}\label{ortho1}
\sum_{n\in \Z}\overline{v_{l, j}(x, n)}v_{l^\prime, j^\prime}(x, n)=0.
\end{align}
Since $\tilde{H}_s(x+p\alpha)=T^{-p}\tilde{H}_s(x)T^p$, where
$T\psi(n)=\psi(n-1),$ we have
\begin{align}\label{defvlj}
v_{l, j}(x+p\alpha, \cdot-p)=v_{l, j+p}(x, \cdot).
\end{align}
Therefore by (\ref{ortho1}) and (\ref{defvlj}), for any $l, l^\prime$, any $p\neq 0$,
\begin{align}\label{ortho2}
\sum_{n}\overline{v_{l, j}(x+p\alpha, n-p)}v_{l^\prime, j}(x, n)=0.
\end{align}

Fix any $l, j$ and $f_q(x)\in L^2(\T)$. Let $\psi_{x}^{q, l, j}(n)=(U_{R}^{-1}U_{k_0}U_Rf_qv_{l, j})(x, n)\in l^2(\Z)$.
Let $\mu_{x}^{q, l, j}$ be the spectral measure of $\tilde{H}_c(x)$ associated to $\psi_x^{q, l, j}(\cdot)$.

\begin{lemma}\label{measureID}
$d\mu_x^{q, l, j}$ is a.e.  independent of $x$.
\end{lemma}
\begin{proof}
Take any continuous function $F$ and $m\neq 0$. By the definition of spectral measure we have, by (\ref{conjHsHc}),
\begin{align*}
  \mathcal{I} \triangleq &|\int_{\T}e^{2\pi i mx} \int F(E)\ d\mu^{q, l, j}_x(E)\ dx|\\
=&|\int_{\T} e^{2\pi i mx}\langle\psi_x^{q, l, j}, F(\tilde{H}_c(x))\psi_x^{q, l, j}\rangle_{l^2(\Z)}\ dx| \notag\\
=&|\langle U_{R}^{-1}U_{k_0}U_R f_q v_{l, j}, e^{2\pi i mx} U_{R}^{-1}U_{k_0}U_RF(\tilde{H}_s) f_q v_{l, j}\rangle_{\mathcal{H}}| \notag
\end{align*}
Applying (\ref{URS}), (\ref{UR-S}) and (\ref{UkS}) to this inner product we have 
\begin{align*}
\mathcal{I}=&|\langle U_{R}^{-1}U_{k_0}U_R f_q v_{l, j}, U_{R}^{-1}S_{-m}U_{k_0}U_RF(\tilde{H}_s) f_q v_{l, j}\rangle_{\mathcal{H}}| \notag\\
=&|\langle U_{k_0}U_R f_q v_{l, j}, e^{2\pi i mk_0(x-\frac{m\alpha}{2})}U_{k_0} S_{-m}U_RF(\tilde{H}_s) f_q v_{l, j}\rangle_{\mathcal{H}}| \notag\\
=&|\langle U_R f_q v_{l, j}, e^{2\pi i mk_0x}S_{-m}U_RF(\tilde{H}_s) f_q v_{l, j}\rangle_{\mathcal{H}}| \notag\\
=&|\langle U_R f_q v_{l, j}, S_{-m}e^{2\pi i mk_0x}U_RF(\tilde{H}_s) f_q v_{l, j}\rangle_{\mathcal{H}}| \notag\\
=&|\langle S_mU_R f_q v_{l, j}, e^{2\pi i mk_0x}U_RF(\tilde{H}_s) f_q v_{l, j}\rangle_{\mathcal{H}}| \notag\\
=&|\langle S_mU_Rf_q v_{l, j}, U_RS_{mk_0}F(\tilde{H}_s) f_q v_{l, j}\rangle_{\mathcal{H}}| \notag\\
=&|\langle U_{R}^{-1}S_mU_R f_q v_{l, j}, S_{mk_0}F(\tilde{H}_s) f_q v_{l, j}\rangle_{\mathcal{H}}| \notag\\
=&|\langle S_{-mk_0}e^{-2\pi i mx} f_q v_{l, j}, F(\tilde{H}_s) f_q v_{l, j}\rangle_{\mathcal{H}}|
\end{align*}
Thus, by (\ref{ortho2}),
\begin{align}
\mathcal{I}
=&|\int_{x\in \T}\sum_{n\in \Z} \overline{e^{-2\pi im(x-mk_0\alpha)}f_q(x-mk_0\alpha)v_{l, j} (x-mk_0\alpha, n+mk_0)} F(e_{l, j}(x)) f_q(x)v_{l, j}(x, n)\ dx| \label{specmeasinde}\\
=&|\int_{x\in \T}e^{2\pi i mx}\overline{f_q(x-mk_0\alpha)}f_q(x)F(e_{l, j}(x)) \sum_{n\in \Z}\overline{v_{l, j}(x-mk_0\alpha, n+mk_0)}v_{l, j}(x, n)\ dx| \notag\\
=&0,  \notag
\end{align}
This result implies $\int F(E)\ d\mu_x^{q, l, j}(E)$ is a.e.  independent of $x$ for all continuous functions $F$. 
Since the set of continuous function is separable, we conclude that $d\mu^{q, l, j}_x$ is a.e.  independent of $x$. $\hfill{} \Box$
\end{proof}

Lemma \ref{measureID} is similar to the analogous (but much simpler) statement in
\cite{gjls} for the Aubry duality transformation $U_R$. After that the
argument of \cite{gjls}  for absolute continuity of the dual measures  relies on the application of Deift-Simon
theorem \cite{ds} (the latter is still unproved for the zero $L(E)$ case leading to a gap in
\cite{gjls}, but correct in case of $L(E)>0.$) Here we however cannot
employ this line of reasoning since  Deift-Simon theorem requires a
second order operator  while our $\tilde{H}_s$ is generally
long-range. Thus we employ a different strategy to obtain absolute
continuity, which has an additional advantage of being somewhat universal.

\begin{lemma}\label{measureac}
For a.e. $x$, $d\mu_x^{q, l, j}$ is absolutely continuous.
\end{lemma}
\begin{proof}
Note that by the definition of spectral measure, for any Borel set $\mathcal{A}$ we have
\begin{align}
  &\int_{\T} d\mu_x^{q, l, j}(\mathcal{A})\ dx \label{twospecmeas}\\
=&\int_{\T}\langle \psi_x^{q, l, j}, \chi_{\mathcal{A}}(\tilde{H}_c(x))\psi_x^{q, l, j}\rangle_{l^2(\Z)}\ dx \notag\\
=&\langle U_{R}^{-1}U_{k_0}U_R f_qv_{l, j}, \chi_{\mathcal{A}}(\tilde{H}_c)U_{R}^{-1}U_{k_0}U_Rf_q v_{l, j}\rangle_{\mathcal{H}}\notag\\
=&\langle U_{R}^{-1}U_{k_0}U_R f_qv_{l, j}, U_{R}^{-1}U_{k_0}U_R \chi_{\mathcal{A}}(\tilde{H}_s)f_qv_{l, j}\rangle_{\mathcal{H}}\notag\\
=&\langle f_qv_{l, j}, \chi_{\mathcal{A}}(\tilde{H}_s)f_qv_{l, j}\rangle_{\mathcal{H}} \notag\\
=&d\mu_{f_qv_{l, j}}(\mathcal{A}), \notag
\end{align}
where $d\mu_{f_qv_{l, j}}$ is the spectral measure of $\tilde{H}_s:
\mathcal{H}\rightarrow \mathcal{H}$ associated to the vector $f_q(x)
v_{l, j}(x, \cdot)\in \mathcal{H}$.
Since $d\mu_x^{q, l, j}$ is a.e.  independent of $x$, by (\ref{twospecmeas}) we get
\begin{align}\label{A}
d\mu_{x}^{q, l, j}(\mathcal{A})=d\mu_{f_qv_{l, j}}(\mathcal{A})\ \mathrm{for}\ \mathrm{a.e.}\ x.
\end{align}
Now it suffices to show that for zero Lebesgue measure set $\mathcal{A}$, $d\mu_{f_qv_{l, j}}(\mathcal{A})=0$. 
Note that again by the definition of spectral measure,
\begin{align}
  &d\mu_{f_qv_{l, j}}(\mathcal{A}) \notag\\
=&\int_{\T}\sum_{n\in \Z}\overline{f_qv_{l, j}(x, n)}(\chi_{\mathcal{A}}(\tilde{H}_s)f_qv_{l, j})(x, n)\ dx. \label{1}
\end{align}
For a.e. $x\in \T$, 
\begin{align}\label{BB}
(\tilde{H}_sv_{l, j})(x, n)=e_{l, j}(x)v_{l, j}(x, n),
\end{align}
thus $(\tilde{H}_sf_qv_{l, j})(x, n)=e_{l, j}(x)f_q(x)v_{l, j}(x, n)$. By (\ref{1}),
\begin{align}
  &d\mu_{f_qv_{l, j}}(\mathcal{A}) \notag\\
=&\int_{\T}\sum_{n\in \Z}\overline{f_q(x)v_{l, j}(x, n)}(\chi_{\mathcal{A}}(\tilde{H}_s)f_qv_{l, j})(x, n)\ dx \notag\\
=&\int_{\T}\sum_{n\in \Z}\overline{f_q(x)v_{l, j}(x, n)}\chi_{\mathcal{A}}(e_{l, j}(x))f_q(x)v_{l, j}(x, n)\ dx \notag\\
=&\int_{\T}\chi_{\mathcal{A}}(e_{l, j}(x))|f_q(x)|^2\ dx. \label{2}
\end{align}
It is thus enough to show for any $q, l, j\in \Z$ that (\ref{2})=0.
We can prove this using the absolutely continuity of the density of states measure.
Note that for any $k\in \Z$, 
\begin{align}\label{3}
dN_{\tilde{H}_s}(\mathcal{A})=\int_{\T}d\mu_{\delta_k, x}(\mathcal{A})\ dx,
\end{align}
where $d\mu_{\delta_k, x}$ is the spectral measure of $\tilde{H}_s(x)$ associated to the vector $\delta_k\in l^2(\Z)$.
Since for a.e. $x$, $v_{l, j}(x, \cdot)$ is an orthonormal basis of $l^2(\Z)$, we have that
\begin{align*}
\delta_k(\cdot)=\sum_{l, j}\langle \delta_k(\cdot), v_{l, j}(x, \cdot) \rangle v_{l, j}(x, \cdot)=\sum_{l, j}v_{l, j}(x, k)v_{l, j}(x, \cdot).
\end{align*}
By (\ref{ortho1}) and (\ref{BB}), this means for a.e. $x$,
\begin{align}\label{4}
d\mu_{\delta_k, x}(\mathcal{A})=\langle \delta_k, \chi_{\mathcal{A}}(\tilde{H}_s(x))\delta_k\rangle=\sum_{l, j}|v_{l, j}(x, k)|^2 \chi_{\mathcal{A}}(e_{l,j}(x)).
\end{align}
By Lemma \ref{IDSac}, $dN_{\tilde{H}_s}(\mathcal{A})=0$. Thus combining (\ref{3}) with (\ref{4}) we get
\begin{align*}
\sum_{l, j}\int \chi_{\mathcal{A}}(e_{l, j}(x))\ dx=\sum_{k\in \Z}\int_{\T}d\mu_{\delta_k, x}(\mathcal{A})\ dx=0.
\end{align*}
This implies in particular
\begin{align*}
\chi_{\mathcal{A}}(e_{l, j}(x))=0\ \mathrm{for}\ \mathrm{a.e.}\ x\ \mathrm{and}\ \mathrm{any}\ l,j.
\end{align*}
By (\ref{2}), $d\mu_{f_qv_{l, j}}(\mathcal{A})=0$.
By (\ref{A}), we conclude that for a.e. $x$, $\{d\mu_{x}^{q, l, j}\}_{q, l, j}$ are absolutely continuous. $\hfill{} \Box$
\end{proof}

Let $\{f_q(x)\}_{q}$ be an orthonormal basis for $L^2(\T)$. Note that the non-vanishing $\{v_{l, j}(x, \cdot)\}$ form an orthonormal basis for $\mathcal{H}$.
It follows that $\{f_q(x)v_{l, j}(x, \cdot)\}_{q, l, j}$ form a complete orthogonal set in $\mathcal{H}$ (but not necessarily orthonormal). 
Since $U_{R}^{-1}U_{k_0}U_R$ is unitary, it follows that $\{(U_{R}^{-1}U_{k_0}U_R f_q v_{l, j})(x, \cdot)\}_{q, l, j}$ is a complete orthogonal set in $\mathcal{H}$.
Thus for a.e. $x$, $\{\psi_{x}^{q, l, j}(\cdot)=(U_{R}^{-1}U_{k_0}U_R f_q v_{l, j})(x, \cdot)\}_{q, l, j}$ is a complete set in $l^2(\Z)$.
Since $\psi_x^{q, l, j}$ is a complete set, we get that $\tilde{H}_c(x)$ only has absolutely continuous spectrum for a.e.  $x$. $\hfill{} \Box$

\section{Absence of eigenvalues. Proof of Theorem \ref{sing} and the second part of Theorem \ref{transition}}\label{singsection}

\subsection{Preparation for the proof of Theorem \ref{sing}}
Consider a general Jacobi operator $(H_{c}(\theta)u)_n=c(\theta+n\alpha)u_{n+1}+\tilde{c}(\theta+(n-1)\alpha)u_{n-1}+v(\theta+n\alpha)u_n$, where $v(\theta)$ is Lipshitz and $c(\cdot)\in \C^\omega(\T)$ is allowed to have zeros on $\T$. 
Let $c(\theta)=f(\theta)g(\theta)$, where $f(\theta)=\prod_{j=1}^m (e^{2\pi i \theta}-e^{2\pi i \theta_j})$ with $\{\theta_j\}_{j=1}^m$ being zeros of $c(\theta)$ counting multiplicities, and $g(\theta)\neq 0$ on $\T$.

Let us define 
\begin{align}\label{defdeltac}
\delta_c(\alpha,\theta)=\limsup_{n\rightarrow\infty}\frac{\sum_{j=1}^m \ln{\|q_n(\theta-\theta_j)\|}+\ln{q_{n+1}}}{q_n}.
\end{align} 
Note that for a.e.  $\theta$, $\delta_c(\alpha,\theta)=\beta(\alpha)$.

We will assume $\theta$ does not belong to the following countable set (otherwise the operator is not well defined).
\begin{align}\label{thetaset}
\theta\notin \Theta\triangleq \cup_{j=1}^m  \theta_j+\Z\alpha+\Z
\end{align}

Fix any $\theta\notin \Theta$ and energy $E$ satisfying $L(E)<\delta_c(\alpha,\theta)$.
Recall that a formal solution to $H_{c}(\theta)u=Eu$ can be reconstructed via the following equation:
\begin{align*}
\left(
\begin{matrix}
u_{n+1}\\
u_n
\end{matrix}
\right)=
A_{c, E}(\theta+n\alpha)
\left(
\begin{matrix}
u_n\\
u_{n-1}
\end{matrix}
\right),
\end{align*}
where 
$A_{c, E}(\theta)=
\left(
\begin{matrix}
\frac{E-v(\theta)}{c(\theta)}\ \ &-\frac{\tilde{c}(\theta-\alpha)}{c(\theta)}\\
1\ \ &0
\end{matrix}
\right)$. 
We separate the singular and regular parts of $A_{c, E}$ and rewrite it in the following way:
\begin{align}\label{defDcE}
A_{c, E}(\theta)=
\frac{1}{f(\theta)}
\left(
\begin{matrix}
\frac{E-v(\theta)}{g(\theta)}\ \ &-\frac{\tilde{c}(\theta-\alpha)}{g(\theta)}\\
f(\theta)\ \ &0
\end{matrix}
\right)
\triangleq
\frac{1}{f(\theta)}D_{c, E}(\theta).
\end{align}
From now on we will omit the dependence of these matrices on $c, E$ and denote $A(\theta)=A_{c, E}(\theta)$ and $D(\theta)=D_{c, E}(\theta)$. Let $A^k=A(\theta+k\alpha)$, $D^k=D(\theta+k\alpha)$. For any function $z(\theta)$ on $\T$ let $z_k=z(\theta+k\alpha)$, for simplicity.
Note that clearly we have $\int_{\T}\ln{|f(\theta)|} d\theta=0$, hence $L(E)=L(\alpha, A)=L(\alpha, D)$.

The first step is standard in Gordon-type methods. For $A\in GL(2, \C)$ we have the following Caley-Hamilton equations:
\begin{align}
A^2-\mathrm{Tr}A\cdot A +\det{A}\cdot \mathrm{Id}=0, \label{B1}\\
A-\mathrm{Tr}A\cdot \mathrm{Id}+\det{A}\cdot A^{-1}=0.  \label{B2}
\end{align}

Fix any $0<\epsilon<(\delta_c(\alpha, \theta)-L(E))/4$. By the definition of $\delta_c(\alpha, \theta)$, there exists a subsequence $\{q_{n_l}\}$ of $\{q_n\}$ such that
\begin{align}\label{C}
\prod_{j=1}^m \|q_{n_l} (\theta-\theta_j)\|\geq \frac{e^{(\delta_c-\epsilon)q_{n_l}}}{q_{n_l+1}}.
\end{align}
We will use the following estimate.
\begin{lemma}\cite{jy} \label{jyest}
\begin{align*}
\prod_{j=0}^{q_{n_l}-1} |f(\theta+j\alpha)| \geq \frac{e^{(\delta_c-\epsilon)q_{n_l}}}{q_{n_l+1}}.
\end{align*}
\end{lemma}

\subsection{Proof of Theorem \ref{sing}}
Assume $u$ is a bounded solution to $H_{c}(\theta)u=Eu$. We could scale $u$ so that 
$\|\left(
\begin{matrix}
u_0\\
u_{-1}
\end{matrix}
\right)\|=1$. We will prove 

\begin{lemma}\label{3vector1large}
For $q_{n_l}$ large enough
\begin{align*}
\max{\left(
\|\left(
\begin{matrix}
u_{q_{n_l}}\\
u_{q_{n_l}-1}
\end{matrix}
\right)\|, 
\|\left(
\begin{matrix}
u_{2q_{n_l}}\\
u_{2q_{n_l}-1}
\end{matrix}
\right)\|, 
\|\left(
\begin{matrix}
u_{-q_{n_l}}\\
u_{-q_{n_l}-1}
\end{matrix}
\right)\|\right) }\geq \frac{1}{4},
\end{align*}
\end{lemma}
\subsubsection*{Proof of Lemma \ref{3vector1large}}
If $
\|\left(
\begin{matrix}
u_{q_{n_l}}\\
u_{q_{n_l}-1}
\end{matrix}
\right)\|=\|A_{q_{n_l}}(\theta)\left(
\begin{matrix}
u_0\\
u_{-1}
\end{matrix}
\right)\|< \frac{1}{4}
$, we divide the discussion into 2 cases:

{\it Case 1}: $|\mathrm{Tr}A_{q_{n_l}}(\theta)|\leq \frac{1}{2}$.

Note that since $|\det{A_{q_{n_l}}(\theta)}|=|\frac{c(\theta-\alpha)}{c(\theta+(q_{n_l}-1)\alpha)}|\rightarrow 1$, (\ref{B1}) implies $\|A_{q_{n_l}}^2(\theta)\left(
\begin{matrix}
u_0\\
u_{-1}
\end{matrix}
\right)\|\geq \frac{7}{8}$ for $q_{n_l}$ large enough.
By telescoping, 
\begin{align}
  &(A_{q_{n_l}}^2(\theta)-A_{2q_{n_l}}(\theta))\left(
\begin{matrix}
u_0\\
u_{-1}
\end{matrix}
\right) \notag\\
=&(\prod_{k=0}^{q_{n_l}-1}A^k-\prod_{k=q_{n_l}}^{2q_{n_l}-1}A^k) A_{q_{n_l}}(\theta)\left(
\begin{matrix}
u_0\\
u_{-1}
\end{matrix}
\right) \notag\\
=&\sum_{i=0}^{q_{n_l}-1} \left(\prod_{k=i+1}^{q_{n_l}-1}A^k \right)\cdot  \left(A^{i}-A^{q_{n_l}+i} \right) \cdot  \left(\prod_{k=q_{n_l}}^{q_{n_l}+i-1} A^k \right) \left(
\begin{matrix}
u_{q_{n_l}}\\
u_{q_{n_l}-1}
\end{matrix}
\right) \notag\\
=&\sum_{i=0}^{q_{n_l}-1} \left(\prod_{k=i+1}^{q_{n_l}-1}A^k \right)\cdot  \left(A^{i}-A^{q_{n_l}+i} \right) \cdot 
\left(
\begin{matrix}
u_{q_{n_l}+i}\\
u_{q_{n_l}+i-1}
\end{matrix}
\right) \notag\\
=&\sum_{i=0}^{q_{n_l}-1} \left(\prod_{k=i+1}^{q_{n_l}-1}\frac{D^k}{f_k} \right)\cdot  
\left( \frac{D_i-D_{q_{n_l}+i}}{f_i} \cdot 
\left(
\begin{matrix}
u_{q_{n_l}+i}\\
u_{q_{n_l}+i-1}
\end{matrix}
\right)
+
\frac{f_{q_{n_l}+i}-f_i}{f_i} \cdot 
\left(
\begin{matrix}
u_{q_{n_l}+i+1}\\
u_{q_{n_l}+i}
\end{matrix}
\right)\right). \label{estimate1}
\end{align}
Note that by our assumption $u$ is a bounded solution, so there exists a constant $C_1>0$ so that
\begin{align}\label{ubdd}
\|\left(
\begin{matrix}
u_{t}\\
u_{t-1}
\end{matrix}
\right)\|\leq C_1\ \mathrm{for}\ \mathrm{any}\ t\in \Z.
\end{align}
Clearly 
\begin{align}\label{fLip}
| f_{q_{n_l}+i}-f_i |\leq \frac{C_2}{q_{n_l+1}}\ \ \mathrm{for}\ \ \mathrm{some}\ \ \mathrm{constant}\ \ C_2,
 \end{align}
and since $v$ is Lipshitz we have
\begin{align}\label{Lip}
\|D_{i}-D_{q_{n_l}+i}\|\leq \frac{C_3}{q_{n_l+1}} \mathrm{for}\ \ \mathrm{some}\ \ \mathrm{constant}\ \ C_3.
\end{align}
Also by Lemmas \ref{jyest}, \ref{uniformupp} and Remark \ref{intcontrolprod}, we have
\begin{align}\label{estimate2}
\frac{\|\prod_{k=i+1}^{q_{n_l}-1}D^k\|}{|\prod_{k=i}^{q_{n_l}-1}f_k|}=\frac{\|\prod_{k=i+1}^{q_{n_l}-1}D^k\|\cdot |\prod_{k=0}^{i-1}f_k|}{|\prod_{k=0}^{q_{n_l}-1}f_k|}\leq q_{n_l+1}e^{(L(E)-\delta_c+3\epsilon)q_{n_l}}.
\end{align}
Now we combine (\ref{estimate1}), (\ref{ubdd}), (\ref{Lip}), (\ref{fLip}) with (\ref{estimate2}),
\begin{align*}
\|(A_{q_{n_l}}^2(\theta)-A_{2q_{n_l}}(\theta))\left(
\begin{matrix}
u_0\\
u_{-1}
\end{matrix}
\right)\|<e^{(L(E)-\delta_c+4\epsilon)q_{n_l}}\rightarrow 0.
\end{align*}
Hence
$\|A_{2q_{n_l}-1}(\theta)\left(
\begin{matrix}
u_0\\
u_{-1}
\end{matrix}
\right)\|\sim 
\|A_{q_{n_l}}^2(\theta)\left(
\begin{matrix}
u_0\\
u_{-1}
\end{matrix}
\right)\|\geq \frac{7}{8}$.

{\it Case 2}: $|\mathrm{Tr}A_{q_{n_l}}(\theta)|> \frac{1}{2}$.

Then $|\det{A_{q_{n_l}}(\theta)}|=|\frac{c(\theta-\alpha)}{c(\theta+(q_{n_l}-1)\alpha)}|\rightarrow 1$ and (\ref{B2}) imply
$\|A_{q_{n_l}}^{-1}\left(
\begin{matrix}
u_0\\
u_{-1}
\end{matrix}
\right)\|\geq \frac{1}{4}$ for $q_{n_l}$ large enough.
Similar to {\it Case 1}, we can prove $\|A_{-q_{n_l}}(\theta)\left(
\begin{matrix}
u_0\\
u_{-1}
\end{matrix}
\right)\|\sim \|A_{q_{n_l}}^{-1}(\theta)\left(
\begin{matrix}
u_0\\
u_{-1}
\end{matrix}
\right)\|\geq \frac{1}{4}$.
$\hfill{} \Box$

By Lemma \ref{3vector1large}, $H_{c}(\theta)$ has no decaying solutions on $\{E: L(E)<\delta_c(\alpha, \theta)\}$, therefore no eigenvalues.
$\hfill{} \Box$

\subsection{Proof of the second part of Theorem \ref{transition}}
Now let's come back to the extended Harper's model, where $c(\theta)=\lambda_1 e^{-2\pi i (\theta+\frac{\alpha}{2})}+\lambda_2+\lambda_3 e^{2\pi i (\theta+\frac{\alpha}{2})}$. In this case, $c(\theta)$ could take zero value when the parameters $\lambda$ satisfy some certain conditions. In fact,
\begin{itemize}
\item when $\lambda_1=\lambda_3\geq \frac{\lambda_2}{2}$, singular points are $\theta_1=\frac{1}{2\pi}\arccos{(-\frac{\lambda_2}{2\lambda_1})}-\frac{\alpha}{2}\ \mathrm{and}\ \theta_2=-\frac{1}{2\pi}\arccos{(-\frac{\lambda_2}{2\lambda_1})}-\frac{\alpha}{2}$ (notice that when $\lambda_1=\frac{\lambda_2}{2}$ there is a single singular point $\theta_1=\theta_2=\frac{1}{2}-\frac{\alpha}{2}$).
\item when $\lambda_1\neq \lambda_3$ and $\lambda_1+\lambda_3=\lambda_2$, the singular point is $\theta_1=\frac{1}{2}-\frac{\alpha}{2}$.
\end{itemize}
Thus for the extended Harper's model, the proper definition of $\delta_c(\alpha, \theta)$ depends on the parameters:
\begin{itemize}
\item {\it Case 1}: (non-singular case): When (1) $\lambda_1\neq \lambda_3$ and $\lambda_1+\lambda_3\neq \lambda_2$ or (2) $\lambda_1=\lambda_3<\frac{\lambda_2}{2}$, we have $\delta_c(\alpha, \theta)=\beta(\alpha)$ for all $\theta$.
\item {\it Case 2}: When $\lambda_1+\lambda_3=\lambda_2$, let $\theta_1\triangleq \frac{1}{2}-\frac{\alpha}{2}$, then $\delta_c(\alpha,\theta)= \limsup_{n\rightarrow\infty}\frac{\ln{\|q_n(\theta-\theta_1)\|}+\ln{q_{n+1}}}{q_n}=\beta(\alpha)$ for a.e. $\theta$.
\item {\it Case 3}: When $\lambda_1=\lambda_3>\frac{\lambda_2}{2}$, let $\theta_1=\frac{1}{2\pi}\arccos{(-\frac{\lambda_2}{2\lambda_1})}-\frac{\alpha}{2}\ \mathrm{and}\ \theta_2=-\frac{1}{2\pi}\arccos{(-\frac{\lambda_2}{2\lambda_1})}-\frac{\alpha}{2}$,  then $\delta(\alpha,\theta)= \limsup_{n\rightarrow\infty}\frac{\sum_{j=1,2}\ln{\|q_n(\theta-\theta_j)\|}+\ln{q_{n+1}}}{q_n}=\beta(\alpha)$ for a.e. $\theta$.
\end{itemize}
Note that for each of the three cases, $L(\lambda)<\beta(\alpha)$
implies absence of eigenvalues for either all or an arithmetic
explicit full measure set of $\theta$. Thus the purely singular continuous part simply comes from the fact that $L(\lambda)>0$.
$\hfill{} \Box$

\section{Pure point spectrum. Proof of the first part of Theorem \ref{transition}}\label{ppsection}
\subsection{Preparation}
Note that if $\lambda=(\lambda_1, \lambda_2, \lambda_3)$ is in region I, its dual $\hat{\lambda}=(\frac{\lambda_3}{\lambda_2}, \frac{1}{\lambda_2}, \frac{\lambda_1}{\lambda_2})$ belongs to region II. By Theorem \ref{LE}, 
for $\lambda$ in region I and any $E\in \Sigma_{\lambda, \alpha}$, $L(E)\equiv L(\lambda)>0$; also for any $E\in \Sigma_{\hat{\lambda}, \alpha}$, $(\alpha, \tilde{A}_{|d|, E})$ is subcritical on $|\mathrm{Im}\theta|\leq \frac{L(\lambda)}{2\pi}$.
It is straightforward that we have $w(d)=0$, since $d(\theta)$ is explicitly given by
\begin{align*}
d(\theta)=\frac{\lambda_1}{\lambda_2}e^{-2\pi i (\theta+\frac{\alpha}{2})}\left(e^{2\pi i (\theta+\frac{\alpha}{2})}-\frac{-1+\sqrt{1-4\lambda_1\lambda_3}}{2\lambda_1}\right)\left(e^{2\pi i (\theta+\frac{\alpha}{2})}-\frac{-1-\sqrt{1-4\lambda_1\lambda_3}}{2\lambda_1}\right).
\end{align*}
The following theorems provide full measure reducibility of $(\alpha, \tilde{A}_{|d|, E})$.
\begin{thm}\label{arc}\cite{arc}
For $\alpha\in \R\backslash \Q$ such that $\beta(\alpha)>0$, if a cocycle $(\alpha, A)$ is subcritical on $|\mathrm{Im}\theta|\leq \frac{h}{2\pi}$, then for every $0<h^\prime<h$, there exists $C>0$ such that if $\delta>0$ is sufficiently small, then there exist a subsequence $\{\frac{p_{n_k}}{q_{n_k}}\}$ of the continued fraction approximants of $\alpha$, sequences of matrices $B_{n_k}\in C^{\omega}_{\frac{h^\prime}{2\pi}}(\T, PSL(2,\R))$ and $R_{n_k}\in SO(2,\R)$ such that $\|B_{n_k}\|_{\frac{h^\prime}{2\pi}}\leq e^{C\delta q_{n_k}}$ and $\|B_{n_k}(\theta+\alpha)A(\theta)B_{n_k}^{-1}(\theta)-R_{n_k}\|_{\frac{h^\prime}{2\pi}}\leq e^{-\delta q_{n_k}}$.
\end{thm}

\begin{thm}\label{reducible}\cite{afk, hy, yz}
Let $(\alpha, A)\in \R\backslash \Q\times C^{\omega}_{\frac{h}{2\pi}}(\T, SL(2,\R))$ with $0<\tilde{h}<h^\prime$, $R\in SO(2,R)$, for every $\tau>0$, $\gamma>0$, if $\mathrm{rot}_f(\alpha, A)\in \mathrm{DC}_{\alpha}(\tau,\gamma)$, where
\begin{align*}
\mathrm{DC}_{\alpha}(\tau, \gamma)=\{\phi\in \T| \|2\phi-m\alpha\|_{\T}\geq \frac{\gamma}{(1+|m|)^{\tau}},\ \mathrm{for}\ \mathrm{any}\ \mathrm{nonzero}\ m\in \Z\}
\end{align*}
then there exists $T=T(\tau)$, $\kappa=\kappa(\tau)$ such that if 
\begin{align}\label{BBB}
\|A(\theta)-R\|_{\frac{h^\prime}{2\pi}}<T(\tau)\gamma^{\kappa(\tau)}(h^\prime-\tilde{h})^{\kappa(\tau)},
\end{align}
then there exists $B\in C^{\omega}_{\frac{\tilde{h}}{2\pi}}(\T, SL(2,\R))$ and $\varphi\in C^{\omega}_{\frac{\tilde{h}}{2\pi}}(\T, \R)$ such that
\begin{align*}
B(\theta+\alpha) A(\theta) B^{-1}(\theta)=R_{\varphi(\theta)},
\end{align*}
with estimates $\|B(\theta)-\mathrm{Id}\|_{\frac{\tilde{h}}{2\pi}}\leq \|A(\theta)-R\|_{\frac{h}{2\pi}}^{\frac{1}{2}}$ and $\|\varphi(\theta)-\hat{\varphi}(0)\|_{\frac{\tilde{h}}{2\pi}}\leq 2 \|A(\theta)-R\|_{\frac{h}{2\pi}}$.
Moreover if $\beta(\alpha)<\tilde{h}$, $(\alpha, A)$ is reducible.
\end{thm}

\subsection{Proof of the pure point part of Theorem \ref{transition}}
This proof follows that of Proposition 4.2 in \cite{ayz}, however some modifications are needed. We include it here for reader's convenience.
Let us consider energy $E\in \Sigma_{\hat{\lambda}, \alpha}$ so that $\rho(\alpha, \tilde{A}_{|d|, E})\in \cup_{\gamma>0}DC_{\alpha}(\tau, \gamma)$ for some $\tau>1$. Note that since $|\cup_{\gamma>0}DC_{\alpha}(\tau, \gamma)|=1$, this is a full density of states measure set of energies.
Fix $\epsilon>0$ small enough so that $\beta(\alpha)<L(\lambda)-2\epsilon$. 

First, by Theorem \ref{arc}, for $h=L(\lambda)$ and $h^\prime=L(\lambda)-\epsilon$, there exists constant $C>0$ so that for $\delta>0$ small there exists a subsequence $\{\frac{p_{n_k}}{q_{n_k}}\}$ of the continued fraction approximants and $B_{n_k}\in C^{\omega}_{\frac{L(\lambda)-\epsilon}{2\pi}}(\T, PSL(2,\R))$, $R_{n_k}\in SO(2,\R)$, such that $\|B_{n_k}\|_{\frac{L(\lambda)-\epsilon}{2\pi}}\leq e^{C\delta q_{n_k}}$ and 
\begin{align}\label{arcarc}
\|B_{n_k}(\theta+\alpha)\tilde{A}_{|d|, E}(\theta)B_{n_k}^{-1}(\theta)-R_{n_k}\|_{\frac{L(\lambda)-\epsilon}{2\pi}}\leq e^{-\delta q_{n_k}}.
\end{align}
As is pointed out in \cite{ayz}, one could consult footnote 5 of \cite{arc} to prove the following estimate on the $\deg{B_{n_k}}$
\begin{align}\label{estdeg}
|\deg{B_{n_k}}| \leq C(\lambda, \epsilon)q_{n_k}
\end{align}
Clearly by (\ref{rhorhorelation}), 
\begin{align}\label{rhoBnkest}
\rho(\alpha, B_{n_k}(\theta+\alpha)\tilde{A}_{|d|, E}(\theta)B^{-1}_{n_k}(\theta))=\rho(\alpha, \tilde{A}_{|d|, E})+\deg{B_{n_k}}\alpha.
\end{align}
Thus since $\rho(\alpha, \tilde{A}_{|d|, E})\in DC_{\alpha}(\tau, \gamma)$ for some $\gamma>0$, by (\ref{estdeg}) and (\ref{rhoBnkest}) we have
\begin{align*}
  &\|\rho(\alpha, B_{n_k}(\theta+\alpha)\tilde{A}_{|d|, E}(\theta)B^{-1}_{n_k}(\theta))+m\alpha\|_{\T}\\
=&\|\rho(\alpha, \tilde{A}_{|d|, E})+(\deg{B_{n_k}}+m)\alpha\|_{\T}\\
\geq &\frac{\gamma}{(1+Cq_{n_k}+|m|)^{\tau}}\\ 
\geq &\frac{(1+Cq_{n_k})^{-\tau}\gamma}{(1+|m|)^{\tau}}.
\end{align*}
This implies $\rho(\alpha, B_{n_k}(\theta+\alpha)\tilde{A}_{|d|, E}(\theta)B^{-1}_{n_k}(\theta))\in DC_{\alpha}(\tau, (1+Cq_{n_k})^{-\tau}\gamma)$.

Secondly, fix $\tilde{h}=L(\lambda)-2\epsilon$. For $q_{n_k}$ large enough, in particular when the following holds, with $T(\tau), \kappa(\tau)$ from (\ref{BBB}),
\begin{align}\label{qnklarge}
(1+Cq_{n_k})^{\tau \kappa(\tau)}<T(\tau) e^{\frac{1}{2}\delta q_{n_k}} (\gamma \epsilon)^{\kappa (\tau)},
\end{align}
we have by (\ref{arcarc})
\begin{align}\label{10}
\|B_{n_k}(\theta+\alpha)\tilde{A}_{|d|, E}(\theta)B_{n_k}^{-1}(\theta)-R_{n_k}\|_{\frac{L(\lambda)-\epsilon}{2\pi}}< T(\tau) (1+Cq_{n_k})^{-\tau\kappa(\tau)}(\gamma\epsilon)^{\kappa(\tau)}.
\end{align}
Thus by Theorem \ref{reducible}, since $\beta(\alpha)<\tilde{h}=L(\lambda)-2\epsilon$, we get $(\alpha, \tilde{A}_{|d|, E})$ is reducible.
Note that this provides us with the requirement to apply our Theorem \ref{main}, and taking into account that $w(d)=0$ we get that $H_{\lambda, \alpha, \theta}$ has pure point spectrum for a.e. $\theta$.   $\hfill{} \Box$

\section*{Acknowledgement}
This research was partially supported by the NSF DMS-1401204. 
We would like to thank Qi Zhou and the anonymous referee for their comments on the earlier versions of this manuscript.
We also gratefully acknowledge support from the Simons Center for Geometry and Physics, Stony Brook University,  where some of this
work was done.

\renewcommand\refname{Reference}


\bibliographystyle{amsplain}

\end{document}